\documentclass[11pt,reqno]{amsart}
\usepackage[parfill]{parskip}    
\usepackage{graphicx}
\usepackage{amsmath}
\usepackage{amssymb}
\usepackage{epstopdf}
\usepackage{pgfplots}
\usepackage{color}
\newtheorem{theorem}{Theorem}[section]
\newtheorem{lemma}[theorem]{Lemma}

\newtheorem{corollary}[theorem]{Corollary}
\newtheorem{remark}[theorem]{Remark}
\newtheorem{definition}[theorem]{Definition}

\author{A. Adali$^1$ and S. Tanveer$^1$} \address{$^1$ Mathematics Department\\The Ohio State University\\Columbus, OH 43210}

\title{Rigorous analytical 
approximation of Tritronqu\'{e}e solution to Painlev\'{e}-I and the first singularity}

\begin{document}
$ $ \vskip -0.2cm
\begin{abstract} 
We use a recently developed method \cite{Costinetal1}, \cite{Costinetal}
to determine approximate expression for tritronqu\'{e}e solution for P-1: $y^{\prime \prime} + 6 y^2 - x=0$ 
in a domain $D$ with rigorous bounds.
In particular we rigorously confirm the location of the closest singularity from the origin to be at
$x= - \frac{770766}{323285} =
-2.3841687675\cdots$ to within $5 \times 10^{-6}$ accuracy, in agreement with previous numerical calculations\cite{Joshi}.      
\end{abstract}

\vskip -2cm
\maketitle

\today

\section{Introduction:}

Painlev\'{e} equations or transcendents, as they are sometimes referred to,
arise quite often in many areas of mathematics; so much so that 
they have sometimes been referred to as nonlinear special
functions. 
While some properties of solutions to Painlev\'{e} equations
are known analytically, there are other properties for which only
numerical evidence exists.
Numerical calculations of these solutions 
still remain an active area of interest (see for instance
\cite{Fornberg},\cite{Fornberg2},\cite{Bassom}).
However, as far as we are aware, numerical computations thus
far have not included rigorous error analysis.
Even for the P-1 equation, for the
special solutions referred to as Tritronq\'{ee} solutions (\cite{Boutroux1}, \cite{Boutroux2}), 
the location of singularity closest to the origin, while computed
with apparent high precision \cite{Joshi}, has not been justified rigorously.  

The purpose of this paper is to (i) determine rigorous error bounds for the
location of the first singularity for 
tritronq\'{ee} 
solution to P-1, and (ii) to demonstrate more generally how a method
developed earlier in the context of 
proof \cite{Costinetal} of the Dubrovin Conjecture for P-1 can be extended
to obtain rigorous error bounds on approximate analytical expression for solution in
different parts of the complex plane. In describing the analysis,
it will also be apparent that the method generalizes 
to other solutions of P-1 and to other differential equations.

\section{Problem:}

The tritronque\'{e} solution\footnote{There are actually five different Tritronquee solutions, each
corresponding to different choice
of anti-Stokes line where one demands $y-\sqrt{\frac{x}{6}} = o(x^{-1/8})$ 
however, they are all related to
each other through rotation of dependent and independent variables.} to the Painlev\'{e}-1 equation:
\begin{equation}
\label{y1}
\frac{d^2y(x)}{dx^2} + 6y^2(x)-x = 0,
\end{equation} 
on $\mathbb{C}$ is the unique solution with the asymptotic behavior
\begin{equation}
\label{y1IC}
y = \sqrt{\frac{x}{6}}\bigg[1 + o\big(x^{-5/8}\big)\bigg] \text{ as } x \rightarrow +\infty.
\end{equation}
It is well-known that any solution to P-1 is single valued and meromorphic  
in $\mathbb{C}$, where singularity locations (in the form of a double poles)
depend on initial conditions on $(y, y^\prime)$. 
Instead of initial conditions, the solution is also completely characterized
by the location of a singularity $x_p$ and the coefficient of ${\hat a}_2$ and has the locally
convergent series representation:
\begin{equation}
\label{eqylocal}
y(x) =-\frac{1}{(x-x_p)^2} + (x-x_p)^2 \sum_{j=0}^\infty {\hat a}_j (x-x_p)^j    
\end{equation}
where ${\hat a}_0=-\frac{x_p}{10} $, ${\hat a}_1=-\frac{1}{6} $, 
${\hat a}_3=0$ and 
for $n \ge 4$, ${\hat a}_n$ is determined from the recurrence relation
\begin{equation}
\label{eqarecur}
{\hat a}_{n} = -\frac{6}{(n+5)(n-2)}\sum_{j=0}^{n-4}{{\hat a}_k {\hat a}_{n-4-k}} 
\end{equation}
The location of $x_p$ closest to the origin for 
the tritronque\'{e} solution is known\cite{Joshi} to be on the negative real axis, and numerical
calculations\cite{Joshi} suggests its location, though this has not been confirmed rigorously.
Singularties at large distance from the origin can be rigorously estimated
based on adiabatic
invariance of conserved quantities\cite{Costin2}. When $x_p$ is not particularly large, 
like the first singularity of the tritronque\'{e} solution, we are unaware of any method of
rigorous analysis to confirm its location.

The main result of this paper is Theorem \ref{thm} given below. However, we first define
a few quantities. 
\begin{definition}
We define $r = \frac{7}{10}$, 
$x_0 = -\frac{770766}{323285}=-2.3841687675\cdots$, 
$L=\frac{11}{2}$, $L_0 = -\frac{49}{100}$,  
$b= \frac{4}{5} (24)^{1/4} $ and  
$a= \frac{5}{2} b $, 
\begin{equation}
\label{eqtau}
\tau= \tau (x) := \frac{x-(L+x_0+r)/2}{(L-x_0-r)/2} 
\end{equation}
and 
$P_u$ and $P$ each to be polynomials:
\begin{equation}
\label{Pu}
P_u (\tau) = \sum_{k=0}^{22} c_k \tau^k \ , ~{\rm where} 
\end{equation}
where ${\bf c} :=\left (c_0, c_1, \cdots, c_{22} \right )$ is given by 
\begin{multline}
\label{Pucoeff}
{\bf c} = \left ( {\frac {335867}{539062}} ,
{\frac {419712}{989125}} ,
 -{\frac {352463}{3539236}} ,
 {\frac {60789}{1703279}} ,
 -{\frac {132842}{11825541}} ,
 {\frac {43961}{54574472}} ,
{\frac {39599}{12036926}} ,
 -{\frac {213665}{48625258}} , \right .\\
\left. {\frac {61644}{14973337}} ,
-{\frac {107283}{33444500}} ,
{\frac {44761}{18892011}} ,
-{\frac {28249}{13550715}} ,
{\frac {20641}{14839893}} ,
{\frac {13459}{92774551}} ,
-{\frac {4992}{34838093}} ,
-{\frac {11771}{8149937}} , \right .\\
\left . {\frac {24115}{27631671}} ,
{\frac {42106}{39550107}} ,
-{\frac {21163}{32637441}},
-{\frac {9782}{15918509}} ,
{\frac {11581}{32652169}} ,
{\frac {14692}{88640147}} ,
-{\frac {12278}{123249611}} \right ) 
\end{multline}
\begin{eqnarray}
P(\zeta) &=& \sum_{k=0}^{17}{a_{n} \zeta^n}  \ ,
\label{P}
\end{eqnarray}
where $a_0=-x_0/10$, $a_1=-1/6$, $a_2=\frac{19949}{321055}$, 
$a_3=0$ and 
\begin{equation}
\label{eqanRecur}
a_n = -\frac{6}{(n+5)(n-2)}\sum_{j=0}^{n-4}{a_k a_{n-4-k}} ~~{\rm for}~~ 17\geq n\geq 4
\end{equation}
Further, we define    
\begin{equation}
\label{eqN0}
N_0 (x) = -\frac{4412401}{98304 \sqrt{6} } x^{-19/2} \left [ 1 - 
\frac{1225}{90049 \sqrt{6}} x^{-5/2} + \frac{30625}{2161176} x^{-5} \right ]
\end{equation}
\begin{equation}
\label{1G1G2}
\mathcal{G}_1 (x) = x^{-5/8} \exp \left [ -i b x^{5/4} \right ] \ , 
\mathcal{G}_2 (x) = x^{-5/8} \exp \left [ i b x^{5/4} \right ] \ , 
\end{equation}  
\begin{equation}
\label{1w0}
w_0 = \sum_{j=1}^{2}\frac{(-1)^{j}}{i a}
\mathcal{G}_{j}(x) \int_{\infty}^{x}{\mathcal{G}_{3-j}(y)y N_0 (y) dy}.
= \Re \left \{ \int_0^\infty e^{-s b x^{5/4}} \mathcal{W}_0 \left (x^{5/4}, s \right ) ds 
\right \}  \ ,  
\end{equation}
where
\begin{equation}
\label{1W0}
\mathcal{W}_0 (z, s) = -\frac{4412401 \sqrt{6}}{368640 a z^7} \left(  \left( 1+i s \right) ^{-15
/2}-{\frac {1225\sqrt{6}}{540294 z^2}}\, \left( 1+ i s
 \right) ^{-19/2}+\frac {30625}{2161176 z^4}\, 
\left( 1+ i s \right)^{-23/2} \right) 
\end{equation}
\begin{remark}
{\rm 
Integration by parts of the Laplace transform representation of $w_0$ and its derivative 
in \eqref{1w0} relates to
error functions with complex arguments, which are considered
known 
in the sense it may be calculated to any desired precision\cite{Abramowitz} 
For instance $w_0 (L) = -1.17414\cdots \times 10^{-7}$, $w_0^\prime (L) = 2.03367\cdots \times 
10^{-7}$.
}
\end{remark}

We also define domains $D_j$ for $j=1,\cdots, 4$ with 
$D_1=[L,\infty)$, $D_2 = [L_0, L)$, $D_3 = [x_0+r,L_0 )$ and 
$D_4=\{x \in \mathbb{C}: |x - x_0| = r, x \neq x_0+r\}$. 
We define $D= D_1\cup D_2 \cup D_3 \cup D_4$ (See Figure \ref{aaa}). 
\end{definition}

\begin{theorem} 
\label{thm}
Let 
\begin{equation}
\label{2y0} 
y_0(x) = \left\{\begin{array}{cc}\sqrt{\frac{x}{6}}\bigg[1+\frac{1}{8\sqrt{6}}x^{-5/2}-\frac{49}{768}x^{-5}+\frac{1225}{1536\sqrt{6}}x^{-15/2}+w_0(x) \bigg] & \text{ on } D_1 \\- \frac{1}{(x-x_0)^2}+P_u(\tau (x)) & \text{ on } D_2\cup D_3 \\-\frac{1}{(x-x_0)^{2}}+(x-x_0)^2 P( x-x_0 ) & \text{ on } D_4 \end{array}\right.
\end{equation}
Then the 
tritronq\'{ee} 
solution $y$ has the representation 
\begin{equation}
y (x) = y_0 (x) + E(x)  
\ , {\rm where} ~~ \Big | E (x) \Big | \le 2.35 \times 10^{-5} \ , \Big |  E^\prime (x) \Big | 
\le 1.16 \times 10^{-4}  
\end{equation}
Moreover, $y$ has a unique double pole 
singularity at $x=x_p \in \left \{ \zeta: \zeta \in \mathbb{C}, |\zeta-x_0| < r \right \}$
with $|x_p - x_0| \leq 4.1 \times10^{-6}$. This is the closest singularity of the 
tritronqu\'{ee} solution from the origin.
\end{theorem}
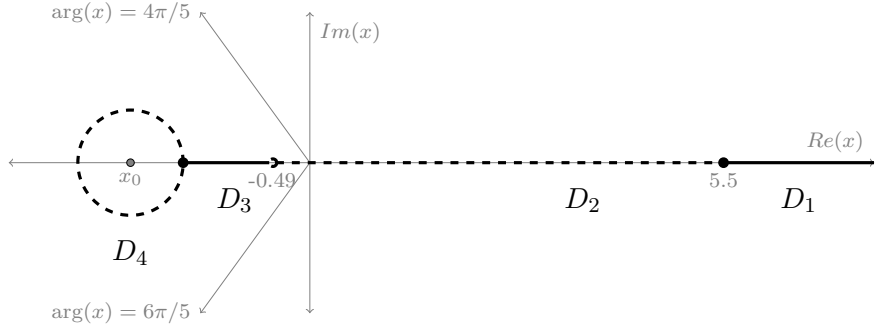
\begin{figure}
\begin{tikzpicture}
    \begin{scope}[thick,font=\scriptsize]
    \draw [<->,very thin,gray] (-4,0) -- (7.5,0) node [above left]  {$Re (x)$};
    \draw [<->,very thin,gray] (0,-2) -- (0,2) node [below right] {$Im (x)$};
    \draw[->,very thin, gray] (0,0)--(-1.453,2) node [left] {$\arg(x) = 4\pi/5$};
    \draw[->,very thin, gray] (0,0)--(-1.453,-2) node [left] {$\arg(x) = 6\pi/5$};
    \draw[very thick](7.5,0)--(5.5,0);
    \draw[very thick, fill=black] (5.5,0) circle (0.05);
    \node[below, gray] at (5.5,0) {5.5};
    \draw[very thick, dashed] (-0.45,0) -- (5.5,0);
    \draw[very thick, dashed] (-0.49,0) circle (0.05);
    \node[below, gray] at (-0.5,0) {-0.49};
    \draw[very thick] (-0.54,0)--(-2.3841687675+0.7,0);
    \draw[very thick,fill=black] (-2.3841687675+0.7,0) circle (0.05);
    \draw[very thick, dashed] (-2.3841687675,0) circle (0.7);
    \node[below, gray] at(-2.3841687675,0) {$x_0$};
     \draw[very thin, fill=gray] (-2.3841687675,0) circle (0.05);
 \end{scope}
    \node[below] at (6.5,-0.2) {$D_1$};
    \node[below right] at (3.25,-0.2) {$D_2$};
    \node[below] at (-1,-0.2) {$D_3$};
    \node[below] at (-2.3841687675,-0.9) {$D_4$};
\end{tikzpicture}
\caption{Sketch of Domain $D = D_1 \cup D_2 \cup D_3 \cup D_4$}
\label{aaa}
\end{figure}
{\bf Note}: The proof of the theorem is completed in Section \ref{sec5} after establishing some preliminary results. 
\begin{remark}
\label{remy0Jump}
{\rm 
From evaluation of expression for $y_0$, $y_0^\prime$ at the points of discontinity, it is readily checked that
$\Big | y_0 (L^+) - y_0 (L^-) \Big | \le 5 \times 10^{-14}$, $\Big | y_0^\prime (L^+) - y_0^\prime (L^-) \Big | \le 7.5 \times 10^{-14} $,
$ \Big | y_0 \left ( [x_0+r]^+ \right ) - y_0 \left ( x_0 + r e^{i 0^+} \right ) \Big | \le 4 \times 10^{-10} $,  
$ \Big | y_0^\prime \left ( [x_0+r]^+ \right ) - y_0^\prime \left ( x_0 + re^{i 0^+} \right ) \Big | \le 7 \times 10^{-8} $
}
\end{remark}  
\begin{remark}
{\rm 
The strategy we pursue is as follows:
For each $D_j$ ($j=1,2,3,4$),
a choice\footnote{The choice relies on  asymptotic series on $D_1$,
Chebyshev polynomial approximation on $D_2\cup D_3$ and Taylor-like polynomial on $D_4$.}
of $y_0$ is made to ensure that the remainder $R(x):=y_0''(x)+6y_0^2(x)-x$ is small and the mismatches
$y_0 (x_e^-) - y_0 (x_e^+)$, $y_0^\prime (x_e^+) - y_0^\prime (x_e^-)$ when boundary point $x_e$ is approached from domains $D_{j}$ and $D_{j+1}$ 
is either zero or very small 
(see Remark \ref{remy0Jump}).
Then the differential equation satisfied by $E = y-y_0$, with initial condition at the end point, is transformed to
a an equivalent nonlinear integral equation in the form $E = \mathcal{N}[E]$.
A convenient Banach space $\mathcal{S} \subset C^0 (D_j) $
is defined and a small enough ball $\mathcal{B}\subset \mathcal{S}$
centered at $E_0 = \mathcal{N} [0]$ for which $\mathcal{N}$ is contractive. 
The Banach contraction mapping theorem ensures existence and uniqueness of such $E \in \mathcal{B}$ satisfying
the integral equation, 
and this result is used to estimate bounds on $|E|$ and $|E'|$. 
In particular, this allows for error analysis in
domain $D_{j+1}$, 
once error estimates are completed in $D_j$ since  
initial conditions in $D_{j+1}$ are then determined to within small errors when continuity of  
$y=y_0+E$ and its derivative is demanded at the end point $x_e$.
\begin{equation}
\label{Ejump}
E(x_e^-) = E(x_e^+) + y_0 (x_e^+) - y_0 (x_e^-) ~~\ , ~~   
E^\prime (x_e^-) = E^\prime (x_e^+) + y_0^\prime (x_e^+) - y_0^\prime (x_e^-) \ ,   
\end{equation}
For each $j$, $C^2 (D_j)$ regularity of $E$ (and therefore of $y$) follows using the smoothness of the Kernel 
in the integral reformulation, and therefore $E$ satisfies the differential equation.
This implies that $y=y_0+E$ satisfies P-1 in each domain $D_j$ and by continuity
of $y$ and $y^\prime$ at common end points $x_e$, it is the same solution to P-1 in
$D$. Since  
asymptotic condition at $\infty$.
has been enforced, this must be 
the trironqu\'{ee} solution.}
\end{remark}

{\bf{Remark on Notation}}:
The framework of the proof in each section is quite similar. Therefore, to avoid proliferation of symbols,
we found it convenient to use the same notation for similar quantities in each section. Thus,
for each subdomain $D_j$ ($j=1, 2, 3, 4$); the approximate solution and the error terms are always denoted by $y_0$ and
$E$ 
respectively\footnote{In \S 2, 
it is more convenient to take $E = \sqrt{\frac{x}{6}} (w-w_0)$ and analyze $w$.}.
$G_1$, $G_2$ will always denote two independent solutions 
of a homogenous second order differential equation $\mathcal{L} E =0$ in each domain,
where $\mathcal{L}$ is typically obtained by 
linearizing P-1 about $y_0$. Sometimes it is more convenient to choose $\mathcal{L}$ as an operator
close to the linearized P-1, but not quite the same. Using $G_1$, $G_2$, the differential
equation for $E$ with initial condition is transformed to an integral equation in the form
$E= \mathcal{N} [E]$.
We always denote
$E_0 = \mathcal{N}[0]$ and $\mathcal{S}$ is the Banach space of
of continuous functions in domain $D_j$ equipped with either sup norm or weighted sup-norm, while
$\mathcal{B} \subset \mathcal{S}$ is the generic notation for some small ball
where $\mathcal{N}$ is contractive. 
Note that the actual definitions of $y_0$, $E$, $G_1$, $G_2$, 
$\mathcal{N}$, $E_0$ , $\mathcal{S}$ and $\mathcal{B}$ differ from section to section. 

\section{Analysis on $D_1 = [L,\infty)$}
It is well-known (for instance see \cite[\S 6.6a]{Costin}) that 
$y$ is uniquely expressed as\footnote{First four terms on an asymptotic series of $y$ is used. 
It is observed that including more or less number of terms result in larger error bounds.} :
\begin{equation}
\label{1y}
y(x) = \sqrt{\frac{x}{6}}\bigg[1+\frac{1}{8\sqrt{6}}x^{-5/2}-\frac{49}{768}x^{-5}+\frac{1225}{1536\sqrt{6}}x^{-15/2}+w(x)\bigg]
\end{equation} 
where $w(x) = o(x^{-15/2})$. Substituting \eqref{1y} into \eqref{y1} and subtracting $[25w(x)]/[64x^2]$ from both sides yield:
 \begin{equation}
\label{1equw}
\mathcal{L} w:=w'' + \frac{w'}{x}+\bigg[2\sqrt{6x}-\frac{25}{64x^2}\bigg]w = N(w(x),x),
\end{equation}
where
\begin{equation}
\label{1N}
 N(w,x) = N_0 (x)  
-\frac{25}{64}x^{-2}\bigg[1-\frac{49}{25\sqrt{6}}x^{-5/2}+\frac{49}{12}x^{-5}\bigg]w-\sqrt{6x}w^2,
\end{equation}
where $N_0 (x)$ is defined in 
\eqref{eqN0}. Note that $N (0, x) = N_0 (x)$.
It is observed that $\mathcal{G}_1$, $\mathcal{G}_2$ defined in 
\eqref{1G1G2}
are independent solutions to the homogenous equation $\mathcal{L} w=0$, defined in
\eqref{1equw}.
From the method of variation of parameter, \eqref{1equw} implies
\begin{equation}
\label{1wequ} 
w(x) = \frac{1}{ia} 
\sum_{j=1}^{2} (-1)^{j} \mathcal{G}_{j}(x) 
\int_{\infty}^x t \mathcal{G}_{3-j} (t) N(w(t),t)dt+\alpha_0 \mathcal{G}_1(x)+\beta_0 \mathcal{G}_2(x),
\end{equation}
\begin{equation*}
=: \mathcal{N} [ w ] (x) + \alpha_0 \mathcal{G}_1 (x) + \beta_0 \mathcal{G}_2 (x)
\end{equation*} 
Since the 
solution sought is in a ball $\mathcal{B}$ in the function space $\mathcal{S}$ (see Definition \ref{defS}),
where $w_0 =\mathcal{N}[0]$ and $\mathcal{N} [w] - w_0$ reside (see (\ref{22}) in the ensuing), 
it follows that $w , \mathcal{N}[w] = O(x^{-10})$, 
implying $\alpha_0=0=\beta_0$ in 
\eqref{1wequ}.
\begin{equation}
\label{1w}
w(x) =\mathcal{N}[w](x).
\end{equation} 
We note that with $w_0$ defined in \eqref{1w0}, $\mathcal{N} [0]=w_0$. 
We define weighted norm $\| . \|_{10}$ on $C(D_1)$:
\begin{equation}
\| f \|_{10} := \sup_{x\in D_1}|f(x)|x^{10} \text{ for } f \in C(D_1). 
\end{equation}
\begin{lemma}
\label{1lemmaw0} $\| w_0 \|_{10} \leq  \frac{15}{2}$ and  $\Big | w_0' (x) \Big | \leq  (8.85) x^{-39/4}$ on $D_1$.
\end{lemma}
\begin{proof} 
Using $x^{5/8} \mathcal{G}_j (x) = \exp \left [ (-1)^j i b x^{5/4} \right ]$, integration by 
by parts in \eqref{1w0} yields:
\begin{equation}
\label{1w0.0}
w_0(x) = \frac{8}{5 ba} x^{-1/2} N_0 (x) + \frac{4}{5ba} 
\sum_{j=1}^2 \mathcal{G}_{j} (x) \int_{x}^\infty \left [ t^{1/8} N_0 (t) \right ]^\prime t^{5/8} \mathcal{G}_{3-j} (t) dt
\end{equation}
From explicit computation, it may be checked that for $x \geq L $, 
$\big|\big[x^{1/8} N_0 \big]'\big| = \big[x^{1/8} N_0 \big]'$, 
$|\mathcal{G}_{j}(x)| \leq x^{-5/8}$ for $j=1,2$ and $x\geq L $ and therefore it may be checked that $x^{19/2} |N_0 (x)|$ is monotonically
increasing for $x \ge L$ and using limiting value as $\lim_{x \rightarrow \infty}$, we obtain
\begin{equation}
\big|w_0(x)\big| 
\le \frac{16}{5ab} x^{-5/8}x^{1/8} |N_0 (x)| \leq \frac{15}{2}x^{-10}, \nonumber
\end{equation}
Further, from 
\eqref{1w0.0},
it follows from using $\mathcal{G}_j (x) \mathcal{G}_{3-j} (x) = x^{-5/4}$ that
\begin{equation*}
\label{eq1w0p}
w_0^\prime = -\frac{1}{ba} x^{-3/2} N_0 (x)  
+\frac{4}{5ba}
\sum_{j=1}^2 \mathcal{G}^\prime_{j} (x) \int_{x}^\infty \left [ t^{1/8} N_0 (t) \right ]^\prime t^{5/8} \mathcal{G}_{3-j} (t) dt
\end{equation*}
Note from \eqref{1G1G2}, 
$|\mathcal{G}^\prime_j(x)| \leq \frac{5}{4} b x^{-3/8}+\frac{5}{8} x^{-13/8}$ for $j=1,2$.
Therefore, again using property $\Big | \left [ t^{1/8} N_0 (t) \right ]^\prime \Big | =
\left [ t^{1/8} N_0 (t) \right ]^\prime $, and $-x^{1/8} N_0 (x) = \Big | x^{1/8} N_0 (x) \Big |$ for $t, x \ge L$,
it follows from 
\eqref{eq1w0p}
and the fact that $x^{19/2} |N_0(x)|=-x^{19/2} N_0(x) $ attains maximal value $18.324\cdots$ at $ x=\infty$, we obtain 
that
\begin{multline}
\big|w_0'(x)\big| \leq 
\frac{1}{ab} x^{-3/2} | N_0 (x) | + 
\frac{8}{5ab}\big[ \frac{5}{4} b x^{-3/8}+ \frac{5}{8} x^{-13/8} \big]
\Big | x^{1/8} N_0 (x) \Big | 
\\
\leq \frac{2}{ax^{39/4}} | x^{19/2} N_0 (x) | \left ( 1 + \frac{1}{b x^{5/4}} \right )       
\leq \frac{(2)(18.33)}{a x^{39/4}} \left (1 + \frac{1}{b L^{5/4}} \right )    
\leq  8.85 x^{-39/4}, \nonumber
\end{multline}
\end{proof}
\begin{definition}
\label{defS}
Define Banach space 
\begin{equation}
\mathcal{S}:= \{ f \in C(D_1): \| f \|_{10} < \infty \},
\end{equation}
For $f_0 \in \mathcal{S}$ and $\epsilon  > 0$, define
\begin{equation}
\mathcal{B}(f_0,\epsilon ):= \{ f \in \mathcal{S}: \| f-f_0 \|_{10} \leq \epsilon   \}.
\end{equation}
\end{definition}
\begin{lemma}
\label{1wuniq}
Let $\delta = 2\times10^{-3}$. 
Then there exists a unique fixed point $w$ of $\mathcal{N}$ in $\mathcal{B}(w_0, \delta \|w_0 \|_{10})$. Moreover, 
$\big|w(x)-w_0(x)\big| \leq (0.0167) \|w_0\|_{10} x^{-45/4}$ and 
$\big|w'(x)-{w_0}'(x)\big| \leq 0.29 x^{-11} $ on $D_1$.
\end{lemma}
\begin{proof} Let $w \in \mathcal{B}(w_0,\delta\|w_0 \|_{10})$. Since 
$\big[1-\frac{49}{25\sqrt{6}}x^{-5/2}+\frac{49}{12}x^{-5}\big] \leq 1$ on $D_1$, $|\mathcal{G}_j(x)| \leq x^{-5/8}$ for $j=1,2$, 
Lemma \ref{1lemmaw0} and $x\geq L$, \eqref{1w} and \eqref{1w0} imply:
\begin{equation}
\big|\mathcal{N}[w](x) - w_0(x) \big| \leq \bigg[ \frac{5 (1+\delta)}{68 a x^{5/4}}  
+\frac{16 \sqrt{6}}{145a x^{35/4}} (1+\delta)^2 \| w_0 \|_{10} \bigg]| \| w_0 \|_{10} x^{-10} 
\leq \delta \| w_0 \|_{10} x^{-10},
 \label{22} 
\end{equation}
Multiplying by $x^{10}$, we immediately obtain
$\mathcal{N}\left [\mathcal{B} \left (w_0,\delta \|w_0 \|_{10} \right ) \right ] \subset \mathcal{B}(w_0,\delta \|w_0 \|_{10})$. 
Similarly, let $w_1  \ , w_2 \in \mathcal{B}(w_0,\delta\|w_0\|_{10})$. 
From the previous arguments, and $\| w_1 \|_{10}+ \| w_2 \|_{10}\leq 2(1+\delta) \|w_0 \|_{10}$, 
\eqref{1w0} implies:
\begin{eqnarray}
\big|\mathcal{N}[w_1] (x) - \mathcal{N}[w_2] (x)\big| &\leq& 
\bigg[\frac{5}{68 a x^{5/4}} +\frac{16\sqrt{6}}{145 a x^{35/4}} 2(1+\delta) \|w_0 \|_{10} \bigg]
x^{-10} \| w_1 - w_2 \|_{10} \nonumber\\
 &\leq&(0.002) x^{-10} \| w_1 - w_2 \|_{10}\nonumber
\end{eqnarray}
Multiplying by $x^{10}$, shows that $\mathcal{N}:\mathcal{B} \rightarrow \mathcal{B}$ is contractive. Thus $w=\mathcal{N} [w]$ has a unique
solution 
$w\in\mathcal{B}(w_0,\delta \|w_0 \|_{10})$ as a consequence of Banach fixed point theorem. Since $w=\mathcal{N} [w]$, it follows 
from estimates in 
\eqref{22} 
that
$\big|w(x)-w_0(x)\big| \leq 0.0167 \|w_0\|_{10} x^{-45/4}$. 
On observing $\big[1-\frac{49}{25\sqrt{6}} t^{-5/2}+\frac{49}{12} t^{-5}\big] \leq 1$ 
on $D_1$ and using \ref{1lemmaw0}, \eqref{1N} and \eqref{eqN0} imply 
$$\big|N(w,t) - 
N_0 (t) \big| \leq \frac{25}{64 t^{12}} \|w_0\|_{10} (1+\delta) + 
\frac{\sqrt{6}}{t^{39/2}} (1+\delta)^2 \| w_0 \|_{10}^2 
\le \frac{2.94}{t^{12}} + \frac{139}{t^{39/2}} 
$$
From this result and the bounds $|\mathcal{G}_j'(x)|  \leq \frac{5b}{4}x^{-3/8} + \frac{5}{8}x^{-13/8}$ for $j=1,2$ for $x \in D_1$,
\eqref{1w}, \eqref{1w0} imply:
 \begin{equation}
\big|w'(x)-{w_0}'(x)\big| \leq   \frac{1}{a} \sum_{j=1}^{2}|\mathcal{G}^\prime_j(x)|
\int^{\infty}_x |t \mathcal{G}_{3-j} | \big|
N(w,t)-N_0 (t)\big|dt \leq 0.29 x^{-11}. \nonumber
\end{equation}
\end{proof}
Recall 
$y_0$ is defined as:
\begin{equation}
\label{1y0}
y_0(x) =: \sqrt{\frac{x}{6}}\bigg[1+\frac{1}{8\sqrt{6}}x^{-5/2}-\frac{49}{768}x^{-5}+\frac{1225}{1536\sqrt{6}}x^{-15/2}
+w_0(x) \bigg],
\end{equation}
On using Lemma \ref{1wuniq}, \eqref{1y0} and \eqref{1y} leads to the main result of this section:
\begin{corollary} 
\label{1yuniq}
The tritronqu\'{ee} solution 
has the representation
$y=y_0 + E$ where 
$\big|E (x) \big| 
\leq (0.00682) \|w_0\|_{10} x^{-43/4}$ and 
$\big|E^\prime (x) \big| \leq(0.126) x^{-21/2}$ on $D_1$. 
In particular, $\Big | E (L^+) \Big | \leq 5.625 \times 10^{-10}$ 
and $|E^\prime (L^+) | \leq (2.12) \times10^{-9}$.
\end{corollary}
\begin{proof}
Because of the smoothness of $\mathcal{G}_1$, $\mathcal{G}_2$ in $D_1$, 
it can be immediately verified that $w$ satisfying
\eqref{1wequ} also  
satisfies the differential equation
\eqref{1equw} with asymptotic condition 
$w(x) = o(x^{-5/8})$ as $x\rightarrow \infty$.
On substituting $y = y_0 + \sqrt{\frac{x}{6}}(w-w_0)$, it follows $y$ satisfies \eqref{y1} 
with the asymptotic condition \eqref{y1IC}. Since such a solution is unique, $y=y_0+\sqrt{\frac{x}{6}} (w-w_0) $ must be the 
tritronqu\'{ee} solution in $D_1$. 
The bounds on the error $E=y(x)-y_0(x) =\sqrt{\frac{x}{6}} (w-w_0)$ and its derivative follow from
those satisfied by $w-w_0$ and its derivative,
and the bounds on $\|w_0\|_{10}$.
\end{proof}
\section{Analysis on $D_2 = [L_0, L )$}
\begin{remark}
\label{2remu}
{\rm 
Since a double pole singularity of $y$ at $x_p$ close to $x_0$ is expected, 
it can potentially cause accuracy problems even on domain $D_2$. 
Hence, it is better to introduce a function $u$ with less variation than $y$ in $D_2$.
}
\end{remark}
We define 
\begin{equation}
\label{2Edef}
E(x) = y(x) - y_0(x),
\end{equation}
Substitution of \eqref{2Edef} into \eqref{y1} yields
\begin{equation}
\label{2eqE}
\mathcal{L} E :=E''+12y_0E = -6E^2 - R,
\end{equation}
where $R$ is the remainder $R(x) := y_0''(x)+6y_0^2(x) - x$. 
Further, imposing continuity of $y_0+E$ at $x=L$ from left and right implies
\begin{equation}
\label{InitE2}
E (L^-) = E(L^+) + y_0 (L^+) - y_0 (L^-)  ~~ \ ,~~ 
E^\prime (L^-) = E^\prime (L^+) + y_0^\prime (L^+) - y_0^\prime (L^-)  , 
\end{equation}
\begin{remark}
\label{remy_0calc}
The explicit form of $y_0$ enables exact evaluation of $y_0$ and all its derivatives, as well as $R$ and
its integral.
\end{remark} 
\begin{definition}
We define 
\begin{equation}
\label{eqcalR}
\mathcal{R}(x) = \int_{L}^{x}{[-R]dt}
\end{equation}
\end{definition}
\begin{lemma}
In the domain $\mathcal{D}_2$, 
\begin{equation}
\label{eqRbound}
\| \mathcal{R} \|_\infty \leq 3.75 \times 10^{-9} 
\end{equation}
\begin{equation}
\label{eqy0bound}
y_0 (x) > 0  ~~~\ ,~~ y_0^\prime (x) > 0
\end{equation} 
\end{lemma}
\begin{proof}
We use a straightfoward inequality (as stated in Lemma \ref{lemA1} in the appendix)  
by partitioning the domain $D_2$ into $n$ equal segments $\left \{ \left [x_{j},x_{j+1} \right ] \right \}_{j=0}^{n-1}$. We note
that both the 
integrals $\int_{x_{j-1}}^{x_j} \left [ \mathcal{R}^\prime (x) \right ]^2 dx$ and point values $\mathcal{R} (x_j)$ can be determined explicitly. Thus
checking bounds on $\| \mathcal{R} \|_\infty$ is easily facilitated resulting in \eqref{eqRbound} by choosing\footnote{$n=20$ gives a bound within about two percent of
the graphically observed bound on $\mathcal{R}$; $n=10$ gives only a $10$ percent higher value.} $n=20$. 
Using Lemma \ref{lemA0}, we have for $x \in [x_{j}, x_{j+1}]$,  $y_0 (x) > \frac{1}{2} \left [ y_0(x_j) + y_0 (x_{j+1}) \right ] - \frac{1}{2} \sqrt{x_j-x_{j-1}} \| y_0^\prime \|_{L^2 [x_{j-1},x_j]} $ and
since the point values of $y_0$ as well as explicit $L_2$ integral can be explicitly determined, we can check condition $y_0 (x) > 0$ in each subinterval.
The same procedure was repeated for $y_0^\prime$ to obtain (\ref{eqy0bound}).
\end{proof}
From \eqref{2eqE}, the two independent
solutions are denoted as  
$G_{1}, G_2$ satisfying
\begin{equation}
\label{2G12}
G_{j}''+12y_0G_{j} = 0, ~~{\rm for} ~~j=1,2
\end{equation}
with initial conditions $G_1 (L^-) = 1$, $G_1'(L^-)=0$, $G_2(L^-)=0$ and $G_2'(L^-)=1$.
\begin{lemma}
\label{2G1G2bd}
$\| G'_1 \|_\infty \leq 3.391$, $\| G_1 \|_\infty \leq 3.775$, $\| G'_2 \|_\infty \leq 1$ and $\| G_2 \|_\infty \leq 
1.114$ on $D_2$.
\end{lemma}

\begin{proof}
On multiplication by $2G_{j}'$, integration from $L$ to $x$ and using integration by parts, \eqref{2G12} gives
\begin{equation}
\label{2G122}
{G_{j}'}^2 (x)  + 12y_0(x)G_{j}^2(x) + 12\int_{x}^{L}{y_0'(t)G_{j}^2 (t) dt} = {\mathcal{G}_j^\prime}^2 (L) + 12 y_0 (L) \mathcal{G}_j^2 (L)  
\end{equation}
Using $y_0 , y_0^\prime > 0$, and initial conditions on $\mathcal{G}_j$, \eqref{2G122} implies
\begin{eqnarray}
\label{2G1bd}
G'_1(x)^2+12y_0(x)G_1(x)^2 &\leq& 12y_0(L),\\
\label{2G2bd}
G'_2(x)^2+12y_0(x)G_2(x)^2 &\leq &1,
\end{eqnarray}
$|G_1'| \leq \sqrt{12y_0(L)}\leq 3.391$ and $|G_2'|\leq 1$ are immediate. 
To find bounds on $G_1$, $G_2$, it is convenient to partition $D_2$ into two intervals $[L_0, \gamma_0)$ and $[\gamma_0, L)$, where
$\gamma_0$ will be chosen appropriately.
Using bounds on $|G_1'|$ in \eqref{2G1bd}, implies
\begin{eqnarray}
|G_1(x)| &\leq& \frac{\sqrt{y_0(L)}}{\sqrt{y_0(x)}} \text{ when } \gamma_0 \leq x \leq L, \nonumber\\
|G_1(x)| &\leq& \int_{x}^{\gamma_0}|G_1'(x)|dx + |G_1(\gamma_0)| \leq (\gamma_0-x)\sqrt{12y_0(L)} +\frac{\sqrt{y_0(L)}}{\sqrt{y_0(\gamma_0)}} \text{ when } L_0 \leq x \leq\gamma_0\nonumber
\end{eqnarray}
Since $y_0$ is monotonically increasing (see \eqref{eqy0bound}), it follows from above that for any $x \in D_2$
\begin{equation} 
\label{G1B}
\Big | G_1 (x) \Big | 
\leq \left (\gamma_0-L_0 \right ) \sqrt{12y_0(L)} +
\frac{\sqrt{y_0(L)}}{\sqrt{y_0(\gamma_0)}} 
\end{equation}
Similarly, using
\eqref{2G2bd}, we obtain for any $x \in D_2$, 
\begin{equation}
\label{G2B}
\Big | G_2 (x) \Big | \le \frac{1}{\sqrt{12 y_0 (\gamma_0)}} + \left (\gamma_0 -L_0 \right )
\end{equation}
From from explicit evaluation with $\gamma_0 = -\frac{16}{100}$, we get
the bounds in the Lemma statement.
\end{proof}

Using variation of parameter, it follows from \eqref{2eqE} that
\begin{equation}
\label{2E}
E(x) = \sum_{j=1}^{2}(-1)^{j+1} G_j(x) \int_{L}^{x}{G_{3-j}(t)[-6E^2 (t)-R(t)]dt} + \alpha_1G_1(x)+\beta_1G_2(x) := \mathcal{N}[E](x)
\end{equation}
where $\alpha_1 = E(L^+) + y_0 (L^+) - y_0(L^-)$ and $\beta_1= E^\prime (L^+) + y_0^\prime(L^+) - y_0^\prime (L^-)$. 
Corollary \ref{1yuniq} and small mismatch of $y_0$ and $y_0^\prime$ (see Remark \ref{remy0Jump}) on two sides of $L$ immediately implies
\begin{corollary}
\label{2alphabetabd}
$|\alpha_1| \leq(5.63)\times10^{-10}$ and $\beta_1\leq (2.13)\times10^{-9}$.
\end{corollary}
\begin{definition}
Define 
\begin{equation}
\label{eqE0}
E_0 (x) = \mathcal{N}[0] (x) =  
\alpha_1 G_1 (x) + \beta G_2 (x) + 
\sum_{j=1}^2 (-1)^{j} G_j (x) \int_L^x G_{3-j} (t) R(t) dt  
\end{equation}
\end{definition}
\begin{lemma}
\label{E0bd}
$\| E_0 \|_\infty \leq (1.745)\times10^{-7} $ and $\| E_0' \|_\infty \leq (1.605)\times10^{-7}$ on $D_2$. 
\end{lemma}
\begin{proof}
Integration by parts and use of boundary conditions at $x=L$ in (\ref{eqE0}) leads to
\begin{equation}
\label{eqParts}
\sum_{j=1}^2 (-1)^{j} G_j (x) \int_L^x G_{3-j} (t) R(t) dt  
= 
\sum_{j=1}^2 (-1)^j G_j (x) \int_{L}^x G^\prime_{3-j} (t) \mathcal{R} (t) dt
\end{equation}
It follows
\begin{eqnarray}
\Big |E_0(x) \Big | &\leq& \big(\|G_1'\|_\infty \| G_2 \|_\infty \|  + \| G_1\|_\infty \| G_2' \|_\infty \big) (L-L_0) 
\| \mathcal{R} \|_\infty + |\alpha_1| \| G_1 \|_\infty + |\beta_1|  
\| G_2 \|_\infty  \leq (1.745) \times10^{-7}, \nonumber \\
\Big | E_0'(x) \Big | &\leq& \| \mathcal{R} \|_\infty 
+ 2 \| G_1' \|_\infty \| G_2' \|_\infty \| \mathcal{R} \|_\infty (L-L_0)+ 
|\alpha_1|. \| G'_1 \|_\infty +|\beta_1|. \|G'_2 \|_\infty \leq 1.605 \times10^{-7},\nonumber
\end{eqnarray}
and taking supremum over $x\in D_2$ gives the result.
\end{proof}
The Banach space $\mathcal{S}$ is defined as
\begin{equation}
\mathcal{S}:= \{ f \in C(D_2):  \|f \|_\infty < \infty \},
\end{equation}
Let $f_0 \in \mathcal{S}$ and $r \geq 0$. The ball $\mathcal{B}(f_0,r) \subset \mathcal{S}$ centered at $f_0$ with radius $r$ is defined as:
\begin{equation}
\mathcal{B}(f_0,r):= \{ f \in \mathcal{S}: \| f-f_0 \|_\infty \leq r  \}.
\end{equation}
\begin{lemma}
\label{2Efixed}
Let $\delta = (5.5)\times10^{-5}$. 
Then there exists a unique fixed point $E$ of $\mathcal{N}$ in $\mathcal{B}(E_0,\delta \|E_0\|_\infty)$. 
Moreover, $ \| E \|_\infty \leq (1.75)\times10^{-7}$ and $\| E' \|_\infty \leq (1.61)\times10^{-7}$. 
\end{lemma}
\begin{proof}
Let $\mathcal{E} \in \mathcal{B} \left (E_0,\delta \|E_0 \|_\infty \right )$. 
Using Lemmas \ref{2G1G2bd} and \ref{E0bd}; \eqref{2E} implies
\begin{equation}
\big| \mathcal{N}[\mathcal{E}](x) - E_0(x) \big| \leq 2\|G_1\|_\infty \| G_2 \|_\infty \left (L-L_0 \right ) 
6(1+\delta)^2 \|E_0\|_\infty^2 \leq \delta \| E_0 \|_\infty \nonumber 
\end{equation}
and taking supremum over $x\in D_2$ gives $\mathcal{N}:\mathcal{B} \rightarrow \mathcal{B}$. 
Let $\mathcal{E}_1, \mathcal{E}_2 \in \mathcal{B}(E_0,\delta \|E_0 \|_\infty)$. From Lemmas \ref{E0bd} and \ref{2G1G2bd}, using
$ \| \mathcal{E}_1 + \mathcal{E}_2 \|_\infty \leq 2(1+\delta) \|E_0\|_\infty$, \eqref{2E} implies:
\begin{equation}
\big| \mathcal{N}[\mathcal{E}_1](x)- \mathcal{N}[\mathcal{E}_2](x) \big| \leq 
2 \| G_1 \|_\infty \| G_2 \|_\infty \left (L-L_0 \right ) 12(1+\delta) \| E_0 \|_\infty   
\leq 1.1 \times 10^{-4}\| \mathcal{E}_1-\mathcal{E}_2 \|_\infty \nonumber
\end{equation}
and taking supremum over $x\in D_2$ 
implies $\mathcal{N}$ is contractive. Banach fixed point theorem implies the existence and uniqueness of such $E$. 
$\| E \|_\infty \leq(1+\delta) \| E_0 \|_\infty \leq (1.75)\times10^{-7}$ 
is immediate. On using  Lemmas \ref{E0bd} and \ref{2G1G2bd}, the derivative of \eqref{2E} leads to
\begin{equation}
|E'(x)|\leq \| E_0'\|_\infty + \big( \| G_1'\|_\infty . \| G_2 \|_\infty + \| G_2'\|_\infty . \| G_1 \|_\infty \big)
(L-L_0) 6 \| E \|_\infty ^2\leq  1.61 \times 10^{-7} . \nonumber 
\end{equation}
and taking supremum over $x\in D_2$ implies the result on $\| E' \|_\infty $.
\end{proof}
\begin{corollary}
\label{2yuniq}
In $D_2$, the tritronqu\'{ee} has the representation $y=y_0+E$ where,   
$ \| E \|_\infty  \leq (1.75)\times10^{-7}$ and $\| E^\prime \|_\infty \leq (1.61)\times10^{-7}$ on $D_2$. 
\end{corollary}
\begin{proof}
The solution $E$ in Lemma \ref{2Efixed} which satifies integral equation
(\ref{2E}) also satisfies 
(\ref{2eqE}) because of smoothness of $G_1$, $G_2$.
Therefore, it immediately follows
that $y=y_0+E$ satisfies P-1. Also the initial conditions 
\eqref{InitE2} on 
$E$ ensures that $y_0+E$ and its derivative at $x=L^-$ match with the tritronqu\'{ee} at $x=L^+$, {\it i.e.} when $L$ is approached from $D_1$.
From uniqueness, $y$ must be the tritronqu\'{ee} solution. Error bounds follow from Lemma \ref{2Efixed}.
\end{proof}

\section{Analysis on $D_3=[x_0+r,L_0)$}

With $y=y_0+E$, where $y_0$ is given by \eqref{2y0}, 
$E$ satisfies
\begin{equation}
\label{3eqE}
E''-\frac{12}{(x-x_0)^2}E = -6E^2 - 12P_uE - R,
\end{equation}
Imposing continuity of $y_0+E$ and its derivative 
as $x=L_0$ is approached from the left and right implies
\begin{equation}
\label{InitE3}
E(L_0^- ) = E(L_0^+) 
~~ \ ,~~ E^\prime (L_0^- ) = E^\prime (L_0^+)  \ ,
\end{equation}
since expression for $y_0$ is the same in domain $D_3$ and $D_2$.
It is observed that
\begin{equation}
\label{3G1G2}
G_1(x) = (x-x_0)^4 \text{ and } G_2(x) = \frac{1}{(x-x_0)^3},
\end{equation}
are two independent solutions of the associated homogenous differential equation $G^{\prime \prime} - \frac{12}{(x-x_0)^2} G=0$.
Using variation of parameter, \eqref{3eqE} implies
\begin{equation}
\label{3E}
E(x) = 
\sum_{j=1}^{2}(-1)^{j}\frac{G_j(x)}{7} \int_{L_0}^{x}{G_{3-j}[-6E^2-12P_uE-R]dt} + \alpha_2G_1(x)+\beta_2G_2(x) := \mathcal{N}[E](x),
\end{equation}
where in order to satisfy (\ref{InitE3}), 
\begin{equation}
\label{3alpha3}
7 \alpha_2=G_2 (L_0) E^\prime (L_0^+ ) 
- G_2' (L_0) E (L_0^+) 
\end{equation} 
\begin{equation}
\label{3beta3}
7 \beta_2=-G_1 (L_0) E^\prime (L_0^+) 
+ G_1^\prime (L_0 ) E (L_0^+)
\end{equation} 
Using bounds in Corollary \ref{2yuniq}, equation \eqref{3alpha3} and \eqref{3beta3} immediately gives
\begin{corollary}
\label{3alphabetabd}
$|\alpha_2|\leq 9.22 \times 10^{-9}=:\alpha_{2,M} $ and $|\beta_2| \leq (9.76)\times10^{-7} =:\beta_{2,M}$.
\end{corollary}
\begin{lemma}
\label{lemPu}
In the domain $D_2 \cup D_3$, $P_u > 0$.
\end{lemma}
\begin{proof}
We note $x \in D_2 \cup D_3$ corresponds to $\tau \in [-1, 1)$. We consider an approximate expression 
\begin{equation}
\label{eqPa}
P_a (\tau) = {\frac {335867}{539062}}+{\frac {835179584688}{1968351794375}}\,\tau-{
\frac {7294680}{73240997}}\,{\tau}^{2}+{\frac {60789}{1703279}}\,{\tau
}^{3}
\end{equation}
It can be checked that 
\begin{equation}
\label{eqPuPa}
P_u (\tau) - P_a (\tau) = \sum_{k=0}^{22} d_k \tau^k  \ , ~{\rm with} ~ \Big | P_u (\tau) - P_a (\tau) \Big |
\le \sum_{k=0}^{22} |d_k| \le 0.039
\end{equation}
On the otherhand, it can be checked
\begin{equation}
P_a^\prime (\tau) = \frac{182367}{1703279}  \left [ \left ( \tau - \frac{40}{43} \right )^2 + \left ( \frac{44}{25}\right )^2
\right ] > 0 
\end{equation}
and therefore $P_a$ is an increasing function implying
\begin{equation}
\label{Palower}
P_a (\tau) > P_a (-1) > 0.063 
\end{equation}
Using 
\eqref{eqPuPa} and \eqref{Palower}, we obtain
\begin{equation}
P_u (\tau) > 0.063 -0.039 > 0  
\end{equation}
\end{proof}
\begin{definition}
Let $\gamma $ be a real number. We define weighted norm $\|. \|_{\gamma}$ on $C(D_3)$:
\begin{equation}
\| f \|_{\gamma}:= \sup_{x\in D_3}\big|(x-x_0)^{\gamma} f(x)\big| \text{ for } f \in C(D_3),
\end{equation}
We also define $E_{0,1}$, $E_{0,2}$ so that $E_0 =\mathcal{N} [0] = E_{0,1} + E_{0,2} $, with
\begin{equation}
E_{0,1} (x) = \alpha_2 G_1 (x) + \beta_2 G_2 (x)  
\end{equation}
\begin{equation}
E_{0,2} (x) = \frac{1}{7} G_1 (x) \int_{L_0}^x G_2 (t) R(t) dt   
-\frac{1}{7} G_2 (x) \int_{L_0}^x G_1 (t) R(t) dt   
\end{equation}
\end{definition}
\begin{remark}
\label{remE01E02}
Since $G_1$, $G_2$ and $R$ are explicit and involve only a finite number of terms in powers of 
$(x-x_0)$, $E_{0,2}$ and $E_{0,2}^\prime$ can
be computed explicitly for any $x\in \mathbb{D}_3$
\end{remark}

\begin{lemma}
\label{3E0bd}
For $\gamma=3.2$,
$\| E_0 \|_{\gamma} \leq 2.04\times10^{-6}$, $\| E_0^\prime \|_\infty \le 1.23 \times 10^{-5}$.
\end{lemma}
\begin{proof} 
Using Corollary \ref{3alphabetabd}, and explicit representation of $G_1$, $G_1^\prime $, $G_2$, 
$G_2^\prime$ it is easy to prove
that 
$(x-x_0)^{\gamma} \left ( \alpha_{2,M} |G_1| + \beta_{2,M} |G_2| \right )
= 
(x-x_0)^{\gamma} \left ( \alpha_{2,M} G_1 + \beta_{2,M} G_2 \right )$, 
attains its maximum value at $x=L_0 $, while
$\alpha_{2,M} |G_1^\prime| + \beta_{2,M} |G_2^\prime | =
\alpha_{2,M} G_1^\prime - \beta_{2,M} G_2^\prime $ 
attains its maximum at $x=x_0+r$.   
Thus, it is readily checked
$\| E_{0,1} \|_{\gamma} \le 2.03 \times 10^{-6} $, 
$\|E_{0,1}^\prime \|_\infty \le 1.2245 \times 10^{-5} $. 
Applying Lemma \ref{lemA1} to 
$(x-x_0)^\gamma E_{0,2}$, and $E_{0,2}^\prime$ for subdivisions of the interval
$[x_0+r, L_0 ]$ with $n=5$, we obtain $\|E_{0,2} \|_\gamma \le 2.3 \times 10^{-9}$, 
$\|E_{0,2}^\prime \|_\infty \le 3.8 \times 10^{-8}$. 
Adding up, we get the bounds for $\| E_0 \|_\gamma$, $\|E_0^\prime \|_\infty$.
given in the Lemma statement.
\end{proof}
\begin{definition}
Define Banach space 
\begin{equation}
\mathcal{S}:= \{ f \in C(D_3): \| f \|_{\gamma} < \infty\},
\end{equation}
For $f_0 \in \mathcal{S}$ and $\epsilon > 0$, define
\begin{equation}
\mathcal{B}(f_0,\epsilon ):= \{ f \in \mathcal{S}: \| f-f_0 \|_{\gamma} \leq \epsilon   \}.
\end{equation}
\end{definition}
\begin{lemma}
\label{3Efixed}
Let $\delta = 0.963$. Then for $\gamma =3.2$, 
there exists a unique fixed point 
$E$ of $\mathcal{N}$ in $\mathcal{B}(E_0,\delta\| E_0 \|_{\gamma})$. 
Moreover, $\| E \|_{\gamma} \leq 4.01 \times10^{-6}$ and 
$\|  E' \|_\infty \leq 3.76 \times10^{-5} $ on $D_3$.
\end{lemma}
\begin{proof}
Let $\mathcal{E} \in \mathcal{B}(E_0,\delta \| E_0 \|_{\gamma})$. 
Using the fact that $G_1, G_2, P_u > 0$ in $D_3$, 
\eqref{3E} implies
\begin{eqnarray}
\big| \big[\mathcal{N}[\mathcal{E}] - E_0\big]\big(x-x_0)^{\gamma} \big| &\leq& \|E \|_{\gamma} \sum_{j=1}^{2}
\frac{(x-x_0)^{\gamma}G_j(x)}{7} \int_{x}^{L_0}{G_{3-j}(t-x_0)^{-\gamma}[12P_u]dt}  \nonumber\\
&+& \| E \|_{\gamma}^2\sum_{j=1}^{2} \frac{(x-x_0)^{\gamma}G_j(x)}{7} \int_{x}^{L_0}{6G_{3-j}(t-x_0)^{-2\gamma}dt}\nonumber
\end{eqnarray}
Since 
$$Q (x) = \sum_{j=1}^2 \frac{(x-x_0)^{\gamma}G_j(x)}{7} \int_{x}^{L_0}{G_{3-j}(t-x_0)^{-\gamma}[12P_u]dt}  \ , $$,
$$T (x) = \sum_{j=1}^2 \frac{(x-x_0)^{\gamma}G_j(x)}{7} \int_{x}^{L_0}{6 G_{3-j}(t-x_0)^{-2\gamma}} dt \ ,  $$ 
and their derivatives can be explicitly calculated, and their upper bounds in $D_2$ determined by applying
Lemma \ref{lemA1} the interval $[x_0+r, L_0 ]$ by partition it into five equal intervals ($n=5$) results in
the bounds
$ \Big | Q (x) \Big | \le 0.49 $ and $\Big | T(x) \Big | \le 1 $     
Therefore, it follows that
\begin{equation}
\big| \big[\mathcal{N}[\mathcal{E}] - E_0\big]\big(x-x_0)^{\gamma} \big| 
\leq 
\bigg[(1+\delta)(0.49)  + (1+\delta)^2\| E_0 \|_{\gamma} \bigg] \| E_0 \|_{\gamma} \leq \delta  \| E_0 \|_{\gamma}
\end{equation} 
and taking supremum over $x\in D_3$ implies $\mathcal{N}:\mathcal{B} \rightarrow \mathcal{B}$. 
Let $\mathcal{E}_1,\mathcal{E}_2 \in \mathcal{B}(E_0,\delta
\| E_0 \|_{\gamma})$. From similar arguments and $\| \mathcal{E}_1 \|_{\gamma}+ \|\mathcal{E}_2 \|_\gamma \leq 
2(1+\delta)\| E_0 \|_{\gamma}$, 
\eqref{3E} yields:
\begin{eqnarray}
\big| \big[\mathcal{N}[\mathcal{E}_1] - \mathcal{N}[\mathcal{E}_2]\big] \big(x-x_0)^{\gamma} \big| &\leq& \bigg[ 
(0.49)  + 
(12.2) (1+\delta) \|E_0 \|_{\gamma} \bigg] \| \mathcal{E}_1-\mathcal{E}_2 \|_{\gamma} \nonumber\\
 &\leq& \frac{1}{2} \times\| \mathcal{E}_1-\mathcal{E}_2 \|_{\gamma}\nonumber
\end{eqnarray}
and taking supremum over $x\in D_3$ implies $\mathcal{N}$ is contractive in $\mathcal{B}$. 
Banach fixed point theorem implies the existence and uniqueness of solution to $E =\mathcal{N} [E]$ in  ball $\mathcal{B}$. 
Further for such a solution,
\begin{eqnarray}
\Big | E'(x) \Big | &\leq&  \| E_0' \|_\infty   + \| E \|_\gamma  \left \{ \sum_{j=1}^{2}
\frac{\Big | G_j'(x) \Big | }{7} \int_{x}^{L_0}{G_{3-j}(t-x_0)^{-\gamma}[12P_u]dt} \right \}  \nonumber\\
&+& \| E \|_{\gamma}^2 \left \{ 6 \sum_{j=1}^{2} \frac{\Big | G'_j(x) \Big |}{7} 
\int_x^{L_0} {G_{3-j}(t-x_0)^{-2\gamma}dt} \right \}
\end{eqnarray}
Using explicit computation of each of the two terms within $\left \{ \cdot \right \}$ at $x=x_0+r$ and
and the bound on $\|E_0^\prime \|_\infty $ in Lemma \ref{3E0bd}, 
$\|  E^\prime \|_\infty \leq \| E_0^\prime \|_\infty + 6.3 \| E \|_\gamma + 12.3 \| E \|_\gamma^2 
\le 3.76 \times 10^{-5} $ 
\end{proof}
The main result of this section is consequence of \ref{3Efixed}:
\begin{corollary}
\label{3yuniq}
In $D_3$, the tritronqu\'{ee} solution has the representation $y=y_0+E$, where 
$ \| E \|_{\gamma}\leq 4.01 \times10^{-6}$ and $\| E^\prime \|_\infty \leq 3.76 
\times10^{-5}$. In particular, \\
$\Big | E (x_0+r) \Big |
\leq 1.26 \times 10^{-5}$ and $\Big |E^\prime (x_0+r) \Big | \leq 3.76 \times 10^{-5} $.
\end{corollary}
\begin{proof}
It is clear from the smoothness of $G_1$, $G_2$ in $D_3$
that the solution $E$ to \eqref{3E} guaranteed by Lemma 
\ref{3Efixed} also satisfies
\eqref{3eqE}, implying $y=y_0+E$ satisfies P-1. The continuity conditions \eqref{InitE3} 
implies that $y_0+E$ matches the tritronqu\'{ee} at $x=x_0+r$. Since the solution to the initial value
problem is unique, $y=y_0+E$ must be the tritronq\'{ee} solution. The error bounds follow from Lemma
\ref{3Efixed}
\end{proof}

\section{Analysis on $D_4=\{ x \in \mathbb{C}: |x - x_0| = r, x \neq x_0+r\}$}

It is useful to introduce $\zeta=x-x_0$. Then, from expression of $y_0$ in Theorem \ref{thm}, it follows that 
\begin{equation}
\label{4y0}
Y_0(\zeta): = y_0 (x_0¤+\zeta) = -\frac{1}{\zeta^2} + \zeta^2 P (\zeta) 
\end{equation}
where $P(\zeta)$ given by \eqref{P}.
$E$ is defined as: 
\begin{equation}
\label{4Edef}
E(\zeta) =y (x_0 + \zeta) - y_0(x_0 + \zeta) 
\end{equation} 
Substitution of \eqref{4Edef} and \eqref{4y0} into \eqref{y1} yields:
\begin{equation}
\label{4eqE}
E''(\zeta)+12Y_0 (\zeta) E(\zeta) = -6E(\zeta)^2 - R(\zeta),
\end{equation}
where $R(\zeta)$ is the residual $R(\zeta):= Y_0''(\zeta) + 6Y_0(\zeta)^2 - (x_0+\zeta)$.
Requiring that as $\nu \rightarrow 0^+$ on $\zeta=r e^{i \nu}$ continuity of
$E(\zeta) + Y_0 (\zeta)$ with solution found in $D_3$, we obtain conditions
\begin{eqnarray}
\label{InitE4} 
E \left ( r e^{i 0^+} \right ) &=& E \left ([x_0+r]^+ \right ) + y_0 \left ( [x_0+r]^+ \right ) - Y_0 \left (r e^{i 0^+} \right )   \\
E^\prime \left ( r e^{i 0^+} \right ) &=& E^\prime \left ([x_0+r]^+ \right ) + y_0^\prime \left ( [x_0+r]^+ \right ) - Y_0^\prime \left (r e^{i 0^+} \right )   ~\ ,~  
\end{eqnarray}
\begin{lemma}
\label{lemR4}
In $D_4$, $\| R \|_\infty \le 1.311 \times 10^{-6} $.
\end{lemma}
\begin{proof}
Recognizing that $Y_0 (\zeta)$ is a truncation of an exact series representation for a solution to P-1 upto a power of $\zeta^{19}$,   
one obtains
\begin{equation}
\label{eqR4}
R(\zeta) = \sum_{j=18}^{38} R_j \zeta^j \ ,  
\end{equation}
implying from calculation of $R_j$ that 
on $|\zeta|=r$,
\begin{equation}
\label{D4Rbound}
\| R \|_\infty \le \sum_{j=18}^{38} |R_j| r^j \le 1.311 \times 10^{-6} 
\end{equation} 
\end{proof}
\begin{remark}
{\rm Now we seek to find bounds on two independent solutions of the homogeneous system $E^{\prime \prime}+12 Y_0 E=0$ associated
with \eqref{4eqE}. For this purpose, it is useful to note the
the geometric bound 
\begin{equation}
\label{Ajbound}
\Big | a_j \Big | \le \frac{1}{2^j} 
\end{equation}
for $0 \le j \le 17$
that can be verified through calculation, using recurrence relation
\eqref{eqanRecur}.
}
\end{remark}

\begin{definition}
\label{D4Gdef}
Define 
\begin{equation}
\label{eq4G1G20}
G_1(\zeta) = \sum_{n=0}^{\infty}{A_n \zeta^{4+n}} \ , 
G_2(\zeta) = \sum_{n=0}^{\infty}{B_n\zeta^{n-3}}, 
\end{equation}
where
\begin{equation}
\label{eq4A}
A_0 = 1 \ , A_{1} =A_2 = A_3 = 0 \ ,
A_{n} = -\frac{12}{n(n+7)}\sum_{k=0}^{\min\{n-4,17\}}{a_{k}A_{n-4-k}}
~~\ , ~~{\rm for} ~n \ge 4
\end{equation}
\begin{equation}
\label{eq4B}
B_0 =1 \ , B_1=B_2=B_3=B_7 =0 \ , ~~{\rm and~for~} n \ge 4, n\ne 7, ~~ 
B_{n} = -\frac{12}{n(n-7)} \sum_{k=0}^{\min \{n-4,17 \}}{a_k B_{n-4-k}}
\end{equation}
\end{definition}
\begin{remark}
\begin{equation}
\label{eq4G1G2}
G_1(\zeta) = \zeta^4 + \sum_{n=0}^{\infty}{A_{n+4} \zeta^{n+8}}  ~~\ ,
~~G_2(\zeta) = 
\zeta^{-3} + \sum_{n=0}^{\infty}{B_{n+4} \zeta^{n+1}} 
\end{equation}
\end{remark}
\begin{lemma}
\label{D4abbound}
For any integer $n \ge 1$,
\begin{equation}
\label{eqABbound}
\Big |A_{n} \Big | \leq c_A \big( \frac{3}{4}\big)^n ~~\ ,~~ 
\Big |B_{n} \Big | \leq  c_B \big( \frac{3}{4}\big)^n \ ,
\end{equation}
where $c_A = 0.21$ and $c_B=0.85$, 
\end{lemma}
\begin{proof}
We checked the inequalities for the first twenty two coefficients $\left \{ A_n, B_n \right \}_{n=1}^{22}$
through explicit calculations 
involving \eqref{eq4A} and \eqref{eq4B}.
Assume the inequality holds for
$n \le n_0$ for some $n_0 \ge 22$. Then, using bounds on $a_k$, and noting that \eqref{eq4A} and \eqref{eq4B} no longer involves
$A_0$, we obtain
\begin{equation}
\Big | A_{n_0+1} \Big |  \le \frac{12}{(n_0+1) (n_0+8)} c_A  
\sum_{k=0}^\infty \left ( \frac{1}{2} \right )^k \left ( \frac{3}{4} \right )^{n_0-3-k}
\le \frac{36 c_A}{(n_0+1)(n_0+8)} \left ( \frac{4}{3} \right )^4 \left ( \frac{3}{4} \right )^{n_0+1} \le
c_A \left ( \frac{3}{4} \right )^{n_0+1} 
\end{equation}
So, the inequality holds for $n_0+1$. By induction it holds for all $n$. 
The same induction proof works for $B_n$ after using
$ \frac{36}{(n_0+1)(n_0-6)} \left ( \frac{4}{3} \right )^4  \le 1 $  
for $n_0 \ge 22$.
\end{proof}
\begin{lemma}
In $D_4$, 
\begin{equation}
\| G_1 \|_\infty \le r^4 + \frac{81 c_A r^8}{256 - 192 r} \le 0.249 ~\ , ~ 
\| G_2 \|_\infty \le \frac{1}{r^3} + \frac{81 c_B r}{256 - 192 r} \le 3.32  
\end{equation}
\begin{equation}
\| G_1^\prime \|_\infty \le 4 r^3 + \frac{81 c_A r^7 (32-21 r)}{64 (4 - 3 r)^2} \le 1.48 ~\ , ~ 
\| G_2^\prime \|_\infty \le \frac{3}{r^4} + \frac{81 c_B}{16 (4-3r)^2} \le 13.7  
\end{equation}
\end{lemma}
\begin{proof}
We use the bounds in \eqref{eqABbound}
in the expression for $G_1$, $G_2$ and their first derivatives obtained from
\eqref{eq4G1G2} and use triangular inequality and the identities
$\sum_{n=0}^{\infty}{|z|^n} = \frac{1}{1-|z|}$,
$\sum_{n=0}^{\infty}{n|z|^{n-1}} = \big(\frac{1}{1-|z|}\big)^2$ for $|z| < 1$.
\end{proof}
\begin{lemma}
\label{D4G1G2bound}
$G_1$ and $G_2$ are two linearly independent solutions to the homogenous differential equation $G'' +12Y_0 G = 0$ on $D_4$ 
with Wronskian $W :=G_1 G_2^\prime - G_2 G_1^\prime = -7$.
\end{lemma}
\begin{proof} 
That $G_1$ and $G_2$ satisfy $G^{\prime \prime} + 12 Y_0 G=0$ is clear 
from substituting each of the series in \eqref{eq4G1G2} for $G$ and noting
that the recurrence relation obtained is the one we have for $A_n$ and $B_n$ in 
\eqref{eq4A} and \eqref{eq4B}, whose geometric decay rate assures convergence in $D_4$. 
Since the Wronskian $W$ must be a constant, it is evident from the small $\zeta$ behavior that
$G_1 = \zeta^4 (1+ o(1) )$, $G_1^\prime = 4\zeta^3 (1+o(1) ) $, $G_2 = \zeta^{-3} (1+ o(1))$, $G_2^\prime =-3 \zeta^{-4} (1+ o(1) )$,
it follows that $W= -7 +o(1)$; since it is independent of $\zeta$, $W=-7$, implying also that the two solutions $G_1$, $G_2$ are independent.
\end{proof}
Applying standard variation of parameter argument, it is clear that
\eqref{4eqE} implies
\begin{equation}
\label{4E}
E(\zeta) = \sum_{j=1}^{2}{(-1)^{j+1}G_j(\zeta)\oint_{x_0+r}^{\zeta}{\frac{G_{3-j}}{7}\big[-6E^2- R\big]dz}}+ \alpha_3G_1(\zeta)+\beta_3G_2(\zeta) := \mathcal{N}[E](\zeta), 
\end{equation}
where the integration $\oint_{x_0}^{\zeta}$ is integral from $x_0$ to $\zeta$ over $D_4$ in the counterclockwise direction, 
$\alpha_3$ and $\beta_3$ are chosen to satisfy \eqref{InitE4}, implying
\begin{eqnarray}
\label{D4alphabeta}
\alpha_3 G_1(r) + \beta_3 G_2(r) &=& E \left ( [x_0+r]^+ \right ) + y_0 \left ( [x_0+r]^+ \right ) - Y_0 \left ( e^{i 0^+} \right ) ,\\
\label{D4alphabeta2}
\alpha_3 G^\prime_1(r) + \beta_3 G^\prime_2(r) &=& E^\prime \left ( [x_0+r]^+ \right ) + y_0^\prime \left ( [x_0+r]^+ \right ) - Y_0^\prime \left ( e^{i 0^+} \right ) ,\\
\end{eqnarray} 
\begin{remark}
{\rm It is to be noted that in the integral reformulation (\ref{4E}) of the differential equation (\ref{4eqE}), we are using the inital
conditions at $\zeta=r$, inherited from the analysis of 
the tritronqu\'{ee} solution in $D_3$. {\it A priori}, we are not requiring that the solution $E$ returns to itself 
when $\nu $ changes from $0$ to $2 \pi$ on the circle $\zeta = r e^{i\nu}$. However, it is well-known that any P-1 solution is single valued, and therefore
so must be the solution $y=y_0+E$ to P-1.
Since $y_0$ is manifestly single valued, so must be $E$.
}
\end{remark}
\begin{lemma}
\label{D4alphabetabound}
$|\alpha_3| \leq 3.82 \times 10^{-5}$ and $|\beta_3| \leq 3.76 \times10^{-6}$.
\end{lemma}
\begin{proof}
We use (\ref{D4alphabeta}) and (\ref{D4alphabeta2}) to solve for $\alpha_3$ and $\beta_3$.
From bounds on
$E^+ (x_0+r)$ and ${E^\prime}^+ (x_0+r)$ in Corollary \ref{3yuniq}, using small
mismatch for 
$y_0$, $y_0^\prime$ at $x_0+r$, 
(see remark \ref{remy0Jump})
and the bounds on
$G_1$, $G_1^\prime$, $G_2$, $G_2^\prime$ from Lemma \ref{D4G1G2bound}, we obtain
the bounds $\alpha_3$ ad $\beta_3$.
\end{proof}
\begin{lemma}
\label{D4E0bound}
Let $E_0$ denote $E_0 = \mathcal{N}[0]$. Then, $\| E_0 \|_\infty \leq 2.34 \times 10^{-5}$ and 
$\| E_0'\|_\infty \leq 1.15 \times10^{-4}$ on $D_4$. 
\end{lemma}
\begin{proof}
Using 
\eqref{4E}
we have
\begin{equation}
\big | E_0 \Big | \le \frac{4 \pi r}{7} \| G_1 \|_\infty  \| G_2 \|_\infty \| R \|_\infty + |\alpha_3 | \| G_1 \|_\infty
+ |\beta_3| \| G_2 \|_\infty \le 2.34 \times 10^{-5}  
\end{equation}   
\begin{equation}
\big | E_0^\prime \Big | \le \frac{2 \pi r}{7} \left ( \| G_1 \|_\infty  \| G_2^\prime \|_\infty 
+ \| G_2 \|_\infty \| G_1^\prime \|_\infty \right ) 
\| R \|_\infty + |\alpha_3 | \| G_1^\prime \|_\infty
+ |\beta_3| \| G_2^\prime \|_\infty \le 1.15 \times 10^{-4}  
\end{equation}   
\end{proof}
\begin{definition}
Define Banach space 
\begin{equation}
\mathcal{S}:= \{ f \in C(D_4): \| f \|_\infty < \infty\},
\end{equation}
For $f_0 \in \mathcal{S}$ and $\epsilon > 0$, define
\begin{equation}
\mathcal{B}(f_0,\epsilon):= \{ f \in \mathcal{S}: \| f-f_0 \| \leq \epsilon \}.
\end{equation}
\end{definition}
\begin{lemma}
\label{4Euniq}
Let $\delta =2 \times 10^{-4} $. There exists a unique fixed point $E$ of $\mathcal{N}$ in $\mathcal{B}(E_0,\delta\|E_0\|_\infty)$. 
Moreover, $ \| E \|_\infty \leq 2.35 \times 10^{-5}$ and $ \| E' \|_\infty \leq 1.16 \times10^{-4}$ on $D_4$.
\end{lemma}
\begin{proof}
Let $\mathcal{E} \in \mathcal{B}(E_0,\delta \|E_0\|_\infty )$.
\begin{eqnarray*}
\Big | \mathcal{N}[\mathcal{E}](\zeta) - E_0(\zeta) \Big | &\leq& 
\frac{24}{7}  \pi r \| G_1 \|_\infty \| G_2 \|_\infty (1+\delta)^2 \| E_0\|_\infty^2 
\leq \delta \| E_0 \| , \end{eqnarray*}
and taking supremum over $\zeta \in D_4$ implies $\mathcal{N}:\mathcal{B} \rightarrow \mathcal{B}$.
Let $\mathcal{E}_1,\mathcal{E}_2 \in \mathcal{B}(E_0,\delta \|E_0 \|_\infty )$. 
From similar arguments and $\| \mathcal{E}_1 \|_\infty + \| \mathcal{E}_2 \|_\infty \leq 2(1+\delta) \|E_0\|_\infty $, \eqref{4E} yields: 
\begin{eqnarray*}
\| \mathcal{N}[\mathcal{E}_1] - \mathcal{N}[\mathcal{E}_2] \|_\infty &\leq&  
\frac{48}{7}  \pi r \| G_1 \|_\infty \| G_2 \|_\infty (1+\delta) \| E_0\|_\infty  \| \mathcal{E}_1 - \mathcal{E}_2 \|_\infty
\leq 3 \times 10^{-4} \| \mathcal{E}_1 - \mathcal{E}_2 \|_\infty \ , \end{eqnarray*}
and therefore, 
$\mathcal{N}$ is contractive. 
Banach fixed point theorem implies the existence and uniqueness of such $E \in \mathcal{B}(E_0,\delta \|E_0\|_\infty )$. The following are immediate:
\begin{eqnarray*}
\| E \|_\infty  &\leq& (1+\delta) \| E_0 \|_\infty  \leq 2.35 \times10^{-5},\\
\| E' \|_\infty   &\leq& \| E_0' \|_\infty + \frac{12 \pi r}{7} \left ( \| G_1'\|_\infty \| G_2 \|_\infty 
+ \| G_2'\|_\infty \| G_1 \|_\infty \right ) (1+\delta)^2 \| E_0 \|_\infty^2 
\leq 1.16 \times10^{-4}.
\end{eqnarray*} 
\end{proof}
\begin{corollary}
\label{4yuniq}
The tritronqu\'{ee} solution to P-1 has the following representation in $D_4$, $y=y_0+E$ where
where $\| E \|_\infty  \leq 2.35 \times10^{-5} $ 
and $ \| E^\prime \|_\infty \leq 1.16 \times10^{-4}$ on $D_4$. 
\end{corollary}
\begin{proof}
Smoothness of $G_1$ and $G_2$ on domain $D_4$ implies that the solution $E$ satisfying \eqref{4E} also solves \eqref{4eqE}.
It is directly checked that $y=y_0+E$ solves P-1 and, because of \eqref{InitE4}, matches the tritronqu\'{ee} at the boundary
point in $D_3$.
Since the solution to P-1 with given initial condition is unique, 
$y=y_0+E$ must be the tritronqu\'{ee}; bounds follow from
Lemma \ref{4Euniq}.
\end{proof}
\section{The proof of the theorem \ref{thm}}
\label{sec5}
\begin{proof}[Proof of Theorem \ref{thm}]
Corollaries \ref{1yuniq}, \ref{2yuniq}, \ref{3yuniq} and \ref{4yuniq} immediate implies the error bounds on $E=y-y_0$ for the 
tritronqu\'{ee}. Since it is well-known that regular solutions of P-1 must be analytic, it follows that $y$ is also analytic
in $D$. Also, any soution to P-1 is meromorphic, with locally convergent series representation \eqref{eqylocal} near each singularity $x_p$. 
Cauchy 
integral formula implies that the integral $-\frac{1}{2\pi i}\oint_{|\zeta|=r}{\zeta y(x_0+\zeta) d\zeta}$ equals 
to the number of singularities of $y(x)$ in $|x-x_0| < r$. From this observation and corollary \ref{4yuniq}, we calculate
\begin{equation}
\label{5eq2}
\bigg|1+\frac{1}{2\pi i }\oint_{|\zeta|=r}{ \zeta y(x_0+\zeta) d\zeta}\bigg| \leq \bigg|1+\frac{1}{2\pi i }\oint_{|\zeta|=r}{ \zeta y_0(x_0+\zeta) d\zeta} \bigg| 
+ r^2 \| y-y_0 \|_\infty \leq 1.2 \times 10^{-5} < 1,
\end{equation}
implying that there is only one singulariy $x_p$ of the tritronqu\'{ee} $y$ in the region $|x-x_0| < r$. 
Once again, Cauchy integral formula and \eqref{eqylocal} imply 
$-\frac{1}{4\pi i}\oint_{|\zeta|=r}{\zeta^2 y(x_0+\zeta) d\zeta} = x_p - x_0$, and
since $\zeta^2 y_0 (x_0+\zeta)$ is analytic in $|\zeta| \le r$, it follows that 
\begin{equation}
\label{5eq4}
|x_p - x_0| \leq \bigg|-\frac{1}{4\pi i }\oint_{|\zeta|=r}{ \zeta^2 E(\zeta) d\zeta} \bigg| \le \frac{r^3}{2} \| E \|_\infty \leq 4.1 \times10^{-6},
\end{equation}
Also, it is known \cite{Joshi} that the pole of the
tritronqu\'{ee} solution closest to the origin is on the negative
real axis. Since our analysis shows solution cannot blow up on $D$, it follows
that the singularity $x_p$, whose location is estimated in \ref{5eq4}, is on
the negative real axis.
\end{proof}

\section{Acknowledgment:} This research was partially supported by the NSF grant 
DMS-1108794.
Additionally, S.T acknowledges support from the Math Departments
,UIC, UChicago
and Imperial College during ST's  sabbatical visit. We also like to acknowledge many helpful
exchanges with Ovidiu Costin. 

\section{Appendix}

\begin{lemma}
\label{lemA0}
Let $y \in C^1 [a,b]$.
then 
\begin{equation}
\label{eqA1}
\| y - \frac{1}{2} \left ( y(a) + y(b) \right ) \|_\infty \le \frac{1}{2} \sqrt{b-a} \| y^\prime \|_{L_2 (a,b)} 
\end{equation}
\begin{equation}
\label{eqA2}
\| y \|_\infty \le \frac{1}{2} \Big | y(a)+y(b) \Big | + \frac{1}{2} \sqrt{b-a} \| y^\prime \|_{L_2 (a,b)}  
\end{equation}
\end{lemma}
\begin{proof}
Note that
$ y(x) = y(a) +\int_a^x y^\prime (t) dt $ and 
$ y(x) = y(b) - \int_x^b y^\prime (t) dt $. 
Therefore,
$$ y(x) = \frac{1}{2} \left ( y(a) + y(b) \right ) + \frac{1}{2} \int_a^x y^\prime (t) dt -\frac{1}{2} \int_x^b 
y^\prime (t) dt $$ 
Therefore, it follows that
$$ \Big | y(x) - \frac{1}{2} \left ( y(a)+y(b) \right ) \Big | \le  \frac{1}{2} \int_a^b |y^\prime (t) | $$
The bounds \eqref{eqA1} immmediately follow from using Cauchy-Schwartz inequality.
The second part \eqref{eqA2} follows from \eqref{eqA1} simply from triangular inequality.
\end{proof}
\begin{lemma}
\label{lemA1}
Let $y \in C^1 [a,b]$, then it follows that 
\begin{equation}
\| y \|_\infty \le \frac{1}{2} \max_{ 0 \le k \le n-1} \left \{ \frac{1}{2} \Big | 
y \left ( a + \frac{k}{n} (b-a) \right ) + y \left ( 1 + \frac{(k+1)}{n} (b-a) \right ) \Big | 
+ \sqrt{\frac{b-a}{n}} \left \{ 
\int_{a+\frac{k}{n} (b-a)}^{a+\frac{(k+1)}{n} (b-a)} {y^{\prime}}^2 (t) dt \right \}^{1/2}   \right \}
\end{equation}
\end{lemma}
\begin{proof} This follows from using previous Lemma on each interval $\left ( a+ \frac{k}{n} (b-a) , 
a + \frac{(k+1)}{n} (b-a) \right )$ and then taking the maximum over the $n$ intervals.
\end{proof}

\end{document}